\documentclass[12pt, reqno]{amsart}
\usepackage{amsmath, amsthm, amscd, amsfonts, amssymb, graphicx, color}
\usepackage[bookmarksnumbered, colorlinks, plainpages]{hyperref}

\makeatletter \oddsidemargin.9375in \evensidemargin \oddsidemargin
\def\Speaker{$^{*}$\protect\footnotetext{ \lowercase{}}}
\def\authorsaddresses#1{\dedicatory{#1}}
\newtheorem{theorem}{Theorem}[section]
\newtheorem{lemma}[theorem]{Lemma}

\newtheorem{corollary}[theorem]{Corollary}
\theoremstyle{definition}

\theoremstyle{remark}

\numberwithin{equation}{section}

\begin{document}

%-----------------------------------------------------------------------------------------------------------------

\title[Laplacian System ]{}{\textbf{ \\\begin{center}a Laplacian System With Sign-Changing Weight Function\end{center}}

\author[ S. S. Kazemipoor AND H. Ebrahimi ] { Seyyed Sadegh Kazemipoor\Speaker And Hadiseh Ebrahimi}
\authorsaddresses{Department of Mathematical Analysis, Charles University, Prague, Czech Republic.\\
}
\thanks{\subjclass MMSC[2010] : {35D30,35j50,35J20}}

\maketitle
%-------------------------------------------------------------------------------
\begin{abstract}
We prove the existence of at least one positive solution for the
Laplacian system
$$\left\{\begin{array}{ll}
-\Delta u=\lambda a(x)|u|^{q-2}u+\beta
\frac{\alpha}{\alpha+\beta}b(x)|u|^{\alpha-2}u|v|^{\beta}&$for~$x\in\Omega$$,  \\
-\Delta v=\lambda a(x)|v|^{q-2}v+\beta\frac{\beta}{\alpha+\beta}b(x)|u|^{\alpha}|v|^{\beta-2}v&$for~$x\in\Omega$$,\qquad(E_{\lambda}) \\
u=v=0 &$for~$x\in\partial\Omega.$$
\end{array}\right.$$
On a bounded region $\Omega$ by using the Nehari manifold and the
fibering maps associated with the Euler functional for the system.
\end{abstract}
\maketitle

%%% ---------------------------------------------------------------------

\section{\noindent\bf Introduction}~ ~In this paper, we study the multiplicity of positive solutions
for the following elliptic system:
\[\left\{\begin{array}{ll}
-div(|\nabla u|^{p-2}\nabla u)=\lambda f(x)|u|^{q-2}u+\frac{r}{r+s}h(x)|u|^{r-2}u|v|^{s} & $in~$\Omega$$,  \\
-div(|\nabla v|^{p-2}\nabla v)=\mu g(x)|v|^{q-2}v+\frac{s}{r+s}h(x)|u|^{r}|v|^{s-2}v & $in~$\Omega$$, \\
u=v=0 &$on~$\partial\Omega.$$
\end{array}\right.\hspace{1cm}(E_{\lambda,\mu})\]
\\
and We shall discuss the
existence of positive solutions of the Laplacian system
\[\left\{\begin{array}{lll}
-\Delta u=\lambda a(x)|u|^{q-2}u+\beta\frac{\alpha}{\alpha+\beta}b(x)|u|^{\alpha-2}u|v|^{\beta}&$for~$x\in\Omega$$,  \\
-\Delta v=\lambda a(x)|v|^{q-2}v+\beta\frac{\beta}{\alpha+\beta}b(x)|u|^{\alpha}|v|^{\beta-2}v&$for~$x\in\Omega$$,\qquad(E_{\lambda}) \\
u=v=0 &$for~$x\in\partial\Omega$$
\end{array}\right.\]
\\
and
$$J_{\lambda}(u^{+},v^{+})=\displaystyle\inf_{(u,v)\in
M_{\lambda}^{+}(\Omega)}J_{\lambda}(u,v)$$ and
$$J_{\lambda}(u^{-},v^{-})=\displaystyle\inf_{(u,v)\in
M_{\lambda}^{-}(\Omega)}J_{\lambda}(u,v).$$ Moreover
$J_{\lambda}(u^{\pm},v^{\pm})=J_{\lambda}(|u^{\pm}|,|v^{\pm}|)$
and $(|u^{\pm}|,|v^{\pm}|)\in M^{\pm}_{\lambda}(\Omega)$ and so
we may assume $u^{\pm}\geq0$ and $v^{\pm}\geq0$. By Lemma 2.2
$(u^{\pm},v^{\pm})$ are critical points of $J_{\lambda}$ on $W$
and hence are weak solutions (and so by standard regularity
results classical solutions) of $(E_{\lambda})$. Finally, by the
Harnack inequality due to Trudinger \cite{21}, we obtain that
$(u^{\pm},v^{\pm})$ are positive solutions of $(E_{\lambda})$
Where $\Omega$ is a bounded region with smooth boundary in $\Bbb
R^{N}$.
Where $r>p,s>p,$~~$1<q<p<r+s<p^{*}$($p^{*}=\frac{pN}{N-p}$ if
$N>p,\quad p^{*}=\infty$ if $N\leq p$), $\Omega\subset\Bbb R^{N}$
is a bounded domain, the pair of parameters $(\lambda,\mu)\in\Bbb
R^{2}-\{(0,0)\}$, and the weight functions $f,g,h\in
C(\overline{\Omega})$ are satisfying $f^{\pm}=\max\{\pm
f,0\}\not\equiv0,\quad g^{\pm}=\max\{\pm g,0\}\not\equiv0, \quad
h^{\pm}=\max\{\pm h,0\}\not\equiv0$.\\
When $p=2$. The fact that number of positive solutions of
equation $(E_{\lambda})$ is affected by the nonlinearity terms has
been the focus of a great deal of research in recent years.\\
If the weight functions $f\equiv h\equiv 1$, the authors
Ambrosetti-Brezis-Cerami\cite{1} have investigated equation
$(E_{\lambda})$. They found that there exists $\lambda_{0}>0$
such that equation $(E_{\lambda})$ admits at least two positive
solutions for $\lambda\in(0,\lambda_{0})$, has a positive
solution for $\lambda=\lambda_{0}$ and no positive solution
exsists for $\lambda>\lambda_{0}$. Wu \cite{7} proved that equation
$(E_{\lambda})$ has at least two positive solutions under the
assumptions the weight functions $f$ change sign in
$\overline{\Omega},\quad g\equiv 1$ and $\lambda$ is sufficiently
small. For more general results, were done by de
Figueiredo-Grossez-Ubilla\cite{5}, Wu\cite{8} and Brown-Wu\cite{2}.\\
In this paper, we give a very simple variational proof which is
similar to proof of Wu (see \cite{9}) to prove the existence of at
least two positive solutions of system $(E_{\lambda,\mu})$ for
$p\in(1,p^{*})$ and so the system $(E_{\lambda,\mu})$ is similar
to the Wu system \cite{10} ( a semilinear elliptic system involving
sign-changing weight functions). In fact, we use the decomposition
of the Nehari manifold as the pair of parameters $(\lambda,\mu)$
varies to prove that the following result.
%%%%%%%%%%%%%%%%%%%%%%%%%%%%%%%%%%%%%%%%%%%%%%%%%%%%%%%%%%%%%%%%%%%%%%
\begin{theorem} There exists $\lambda_{0}>0$ and
$\mu_{0}>0$ such that for $0<|\lambda|<\lambda_{0}$ and
$0<|\mu|<\mu_{0}$, system $(E_{\lambda,\mu})$ has at least two
positive solutions.
\end{theorem}

This paper is organized as follows. In section 2, we give some
notations and preliminaries. In section 3, we establish the
existence of Palais-Smale sequences and we prove the
system  $(E_{\lambda,\mu})$ has at least two positive solutions.

\section{\noindent\bf  Notations and Preliminaries}
~~Throughout this paper, we denote by $S_{l}$ the best Sobolev
constant for the operators $ W=W^{1,p}_{0}(\Omega)\times
W^{1,p}_{0}(\Omega) \hookrightarrow L={L^{l}(\Omega)}\times
{L^{l}(\Omega)}$ is given by
$$S_{l}=\displaystyle\inf_{(u,v)\in W-\{(0,0)\}}\frac{(\int_{\Omega}|\nabla u|^{p}+\int_{\Omega}|\nabla v|^{p})}{(\int_{\Omega}|u|^{l}+\int_{\Omega}|v|^{l})^{\frac{p}{l}}}>0$$
By the above results we know that the existence and multiplicity
of positive solutions of semilinear elliptic problems depend on
the nonlinearity term . Let $S$ be the best Sobolev constant for the embedding of
$H_{0}^{1}(\Omega)$ in
$L^{\alpha+\beta}(\Omega)$. Then we have the following result
Where $1<l\leq p^{*}$. In particular,
$(\int_{\Omega}|u|^{l}+\int_{\Omega}|v|^{l})\leq
S^{-\frac{l}{p}}_{l}||(u,v)||^{l}$ for all $(u,v)\in W$ with the
standard norm $||(u,v)||=(\int_{\Omega}|\nabla
u|^{p}+\int_{\Omega}|\nabla v|^{p})^{\frac{1}{p}}$.\\
System $(E_{\lambda,\mu})$ is posed in the framework of the
Sobolev space $W$. Moreover, a function $(u,v)\in W$ is said to
be a weak solution  of system $(E_{\lambda,\mu})$ if
$$\int_{\Omega}|\nabla u|^{p-2}\nabla
u\nabla\varphi-\lambda\int_{\Omega}f|u|^{q-2}u\varphi-\frac{r}{r+s}\int_{\Omega}h|u|^{r-2}u|v|^{s}\varphi=0$$
for all $\varphi\in W_{0}^{1,p}(\Omega)$ and
$$\int_{\Omega}|\nabla v|^{p-2}\nabla
v\nabla\varphi-\mu\int_{\Omega}g|v|^{q-2}v\varphi-\frac{s}{r+s}\int_{\Omega}h|u|^{r}|v|^{s-2}v\varphi=0$$
for all $\varphi\in W_{0}^{1,p}(\Omega)$. Thus, the corresponding
energy functional of system $(E_{\lambda,\mu})$ is defined by
$$\begin{array}{rcl}
J_{\lambda,\mu}(u,v)&=&\frac{1}{p}||(u,v)||^{p}-\frac{1}{q}(\lambda\int_{\Omega}f|u|^{q}\\
&+&\mu\int_{\Omega}g|v|^{q})-\frac{1}{r+s}(\int_{\Omega}h|u|^{r}|v|^{s})\qquad
for~~(u,v)\in W$$
\end{array}$$
As the energy functional $J_{\lambda,\mu}$ is not bounded below
on $W$, it is useful to consider the functional on the Nehari
manifold $$M_{\lambda,\mu}=\{(u,v)\in W-\{(0,0)\}\mid\langle
J^{'}_{\lambda,\mu}(u,v),(u,v)\rangle=0\}$$ Thus, $(u,v)\in
M_{\lambda,\mu}$ if and only if
$$||(u,v)||^{p}-(\lambda\int_{\Omega}f|u|^{q}+\mu\int_{\Omega}g|v|^{q})-(\int_{\Omega}h|u|^{r}|v|^{s})=0\qquad(1)$$
Define $$\begin{array}{rcl}\psi_{\lambda,\mu}(u,v)=\langle
J^{'}_{\lambda,\mu}(u,v),(u,v)\rangle&=&||(u,v)||^{p}-(\lambda\int_{\Omega}f|u|^{q}+\mu\int_{\Omega}g|v|^{q})\\&-&\int_{\Omega}h|u|^{r}|v|^{s}.\\
\end{array}$$
Then for $(u,v)\in M_{\lambda,\mu}$,
$$\begin{array}{rcl}\langle\psi^{'}_{\lambda,\mu}(u,v),(u,v)\rangle&=&p||(u,v)||^{p}-q(\lambda\int_{\Omega}f|u|^{q}+\mu\int_{\Omega}g|v|^{q})\\
&-&(r+s)\int_{\Omega}h|u|^{r}|v|^{s}\\
&=&(p-q)(\lambda\int_{\Omega}f|u|^{q}+\mu\int_{\Omega}g|v|^{q})\\
&+&(p-r-s)\int_{\Omega}h|u|^{r}|v|^{s},\qquad(2)\\
\end{array}$$ Now, we split $M_{\lambda,\mu}$ into three parts:
$$\begin{array}{rcl}M_{\lambda,\mu}^{+}&=&\{(u,v)\in
M_{\lambda,\mu}\mid\langle\psi^{'}_{\lambda,\mu}(u,v),(u,v)\rangle>0\}\\
M_{\lambda,\mu}^{0}&=&\{(u,v)\in
M_{\lambda,\mu}\mid\langle\psi^{'}_{\lambda,\mu}(u,v),(u,v)\rangle=0\}\\
M_{\lambda,\mu}^{-}&=&\{(u,v)\in
M_{\lambda,\mu}\mid\langle\psi^{'}_{\lambda,\mu}(u,v),(u,v)\rangle<0\}\\
\end{array}$$
As $J_{\lambda}$ is not bounded below on $W$, it is useful to
consider the functional on the Nehari manifold
$$M_{\lambda}(\Omega)=\{(u,v)\in W~:~~\langle J^{'}_{\lambda}(u,v),(u,v)\rangle=0\}$$
where $\langle~,~\rangle$ denotes the usual duality. Thus
$(u,v)\in M_{\lambda}(\Omega)$ if and only if
$$\begin{array}{rcl}&&(\int_{\Omega}|\nabla u|^{2}dx+\int_{\Omega}|\nabla
v|^{2}dx)-\lambda(\int_{\Omega}a(x)|u|^{q}dx+\int_{\Omega}a(x)|v|^{q}dx)\\
&-&(\int_{\Omega}b(x)|u|^{\alpha}|v|^{\beta}dx)=0.\hspace{4cm}(2.1)\\
\end{array}$$
Clearly $M_{\lambda}(\Omega)$ is a much smaller set than $W$ and, as
we shall show, $J_{\lambda}$ is much better behaved on
$M_{\lambda}(\Omega)$. In particular, on $M_{\lambda}(\Omega)$ we
have that
$$\begin{array}{rcl}J_{\lambda}(u,v)&=&(\frac{1}{2}-\frac{1}{q})(\int_{\Omega}|\nabla u|^{2}dx+\int_{\Omega}|\nabla
v|^{2}dx)\\
&+&(\frac{1}{q}-\frac{1}{\alpha+\beta})(\int_{\Omega}b(x)|u|^{\alpha}|v|^{\beta}dx)\\
&=&(\frac{1}{2}-\frac{1}{\alpha+\beta})(\int_{\Omega}|\nabla u|^{2}dx+\int_{\Omega}|\nabla
v|^{2}dx)\\
&-&\lambda(\frac{1}{q}-\frac{1}{\alpha+\beta})(\int_{\Omega}a(x)|u|^{q}dx+\int_{\Omega}a(x)|v|^{q}dx).\\
\end{array}\hspace{1cm}(2.2)$$\
Then, we have the following results.\\\\
%%%%%%%%%%%%%%%%%%%%%%%%%%%%%%%%%%%%%%%%%%%%%%%%%%%%%%%%%%%%%%%%%%%%%%%%%

\begin{lemma}\label{maintheorem0} If $(u_{0},v_{0})$ is a local minimizer for
$J_{\lambda,\mu}$ on $M_{\lambda,\mu}$ and $(u_{0},v_{0})\notin
M_{\lambda}^{0}$, then $J_{\lambda,\mu}^{'}(u_{0},v_{0})=0$ in
$W^{'}=W^{-1,p^{'}}(\Omega)\times W^{-1,p^{'}}(\Omega)$
existence of positive solutions of the Laplacian system
\[\left\{\begin{array}{lll}
-\Delta u=\lambda a(x)|u|^{q-2}u+\beta\frac{\alpha}{\alpha+\beta}b(x)|u|^{\alpha-2}u|v|^{\beta}&$for~$x\in\Omega$$,  \\
-\Delta v=\lambda a(x)|v|^{q-2}v+\beta\frac{\beta}{\alpha+\beta}b(x)|u|^{\alpha}|v|^{\beta-2}v&$for~$x\in\Omega$$,\qquad(E_{\lambda}) \\
u=v=0 &$for~$x\in\partial\Omega$$
\end{array}\right.\]
\\
.
\end{lemma}
\begin{proof}[\bf Proof]
Our proof is almost the same  as that in Brown-Zhang[\cite{3}, theorem
2.3].
\end{proof}
\begin{lemma}\label{maintheorem8}
The energy functional $J_{\lambda,\mu}$ is coercive and bounded
below on $M_{\lambda,\mu}$.
\end{lemma}
\begin{proof}[\bf Proof] If $(u,v)\in M_{\lambda,\mu}$, then by the Sobolev
trace imbedding theorem
$$\begin{array}{rcl}J_{\lambda,\mu}(u,v)&=&\frac{1}{p}||(u,v)||^{p}-\frac{1}{q}(\lambda\int_{\Omega}f|u|^{q}+\mu\int_{\Omega}g|v|^{q})\\
&-&\frac{1}{r+s}\int_{\Omega}h|u|^{r}|v|^{s}\\
&=&\frac{r+s-p}{p(r+s)}||(u,v)||^{p}+(-\frac{1}{q}+\frac{1}{r+s})(\lambda\int_{\Omega}f|u|^{q}+\mu\int_{\Omega}g|v|^{q})\\
&\geq&
\frac{r+s-p}{p(r+s)}||(u,v)||^{p}+S_{q}^{\frac{q}{p}}(\frac{r+s-q}{q(r+s)})(|\lambda|||f||_{\infty}||u||^{q}+|\mu|||g||_{\infty}||v||^{q})\\
\end{array}$$ Thus, $J_{\lambda}$ is coercive and bounded below on
$M_{\lambda,\mu}$.
\end{proof}
\begin{lemma}\label{maintheorem3}
(i) For any $(u,v)\in M_{\lambda,\mu}^{+}$, we have
$(\lambda\int_{\Omega}f|u|^{q}+\mu\int_{\Omega}g|v|^{q})>0$
$$\begin{array}{rcl}
&(ii)& For~ any ~(u,v)\in M_{\lambda,\mu}^{0}, we~ have~
(\lambda\int_{\Omega}f|u|^{q}+\mu\int_{\Omega}g|v|^{q})>0\\
&and& \int_{\Omega}h|u|^{r}|v|^{s}>0\\
&(iii)& For ~any ~(u,v)\in M_{\lambda,\mu}^{-}, we ~have~
 \int_{\Omega}h|u|^{r}|v|^{s}>0.\\
\end{array}$$
\end{lemma}
\begin{proof}[\bf Proof] The proofs are immediate from (1) and
(2).
\end{proof}
\begin{lemma}\label{maintheorem5} The exists $\lambda_{0}>0, \mu_{0}>0$ such that for
$0<|\lambda|<\lambda_{0}$, $0<|\mu|<\mu_{0}$ we have $M_
{\lambda,\mu}^{0}=\emptyset$.
\end{lemma}
\begin{proof}[\bf Proof] Suppose otherwise, that is $M_ {\lambda,\mu}^{0}\neq\emptyset$ for
all $(\lambda,\mu)\in\Bbb R^{2}-\{(0,0)\}$. Then by lemma \ref{maintheorem3},
$$\begin{array}{rcl}
0&=&\langle
J_{\lambda,\mu}^{'}(u,v),(u,v)\rangle=(p-q)||(u,v)||^{p}\\
&-&(r+s-q)\int_{\Omega}h|u|^{r}|v|^{s}\\
&=&(p-r-s)||(u,v)||^{p}-(q-r-s)(\lambda\int_{\Omega}f|u|^{q}+\mu\int_{\Omega}g|v|^{q})\\
\end{array}$$
for all $(u,v)\in M_{\lambda,\mu}^{0}$. By the H\~{o}lder
inequality, Minkowski inequality and the Sobolev imbedding
theorem,
$$||(u,v)||\geq(\frac{2(p-q)}{r+s-q}||h||_{\infty}S_{r+s}^{\frac{r+s}{p}})^{\frac{1}{(p-r-s)}}$$
and
$$||(u,v)||\leq(\frac{r+s-q}{r+s-p}(|\lambda|~||f||_{\infty}+|\mu|~||g||_{\infty}))^{\frac{1}{p-q}}S_{p}^{\frac{q}{p(p-q)}}$$
If $|\lambda|,|\mu|$ is sufficiently small, this is impossible.
Thus, we can conclude that there exists $\lambda_{0}>0, \mu_{0}>0$
such that if $0<|\lambda|<\lambda_{0}$ and $0<|\mu|<\mu_{0}$, we
have $M_ {\lambda,\mu}^{0}=\emptyset$.
\end{proof}
\begin{theorem} $J_{\lambda}$ is coercive and bounded below on $M_{\lambda}(\Omega)$.
\end{theorem}
\begin{proof}[\bf Proof]
$$\left\{\begin{array}{ll}
-\Delta u=\lambda a(x)|u|^{q-2}u+\beta
\frac{\alpha}{\alpha+\beta}b(x)|u|^{\alpha-2}u|v|^{\beta}&$for~$x\in\Omega$$,  \\
-\Delta v=\lambda a(x)|v|^{q-2}v+\beta\frac{\beta}{\alpha+\beta}b(x)|u|^{\alpha}|v|^{\beta-2}v&$for~$x\in\Omega$$,\qquad(E_{\lambda}) \\
u=v=0 &$for~$x\in\partial\Omega.$$
\end{array}\right.$$
It follows from (2.2) and by the H$\ddot{o}$lder inequality
$$\begin{array}{rcl}J_{\lambda}(u,v)&=&\frac{\alpha+\beta-2}{2(\alpha+\beta)}||(u,v)||^{2}_{H}-(\frac{\alpha+\beta-q}{q(\alpha+\beta)})(\int_{\Omega}\lambda a(x)|u|^{q}+\int_{\Omega}\lambda|v|^{q}dx)\\
&\geq&\frac{\alpha+\beta-2}{2(\alpha+\beta)}||(u,v)||^{2}_{H}\\
&-&S^{q}(\frac{\alpha+\beta-q}{q(\alpha+\beta)})((|\lambda|||a||_{L^{p^{*}}})^{\frac{2}{2-q}}+(|\lambda|||a||_{L^{p^{*}}})^{\frac{2}{2-q}})^{\frac{2-q}{2}}||(u,v)||^{q}_{H}\\
&\geq&\frac{S^{\frac{2q}{2-q}}(q-2)(\alpha+\beta-q)^{\frac{2}{2-q}}}{2q(\alpha+\beta)(\alpha+\beta-2)^{\frac{q}{2-q}}}((|\lambda|||a||_{L^{p^{*}}})^{\frac{2}{2-q}}+(|\lambda|||a||_{L^{p^{*}}})^{\frac{2}{2-q}}).\\
\end{array}$$
Thus, $J_{\lambda}$ is coercive on $M_{\lambda}(\Omega)$ and
$$J_{\lambda}(u,v)\geq\frac{S^{\frac{2q}{2-q}}(q-2)(\alpha+\beta-q)^{\frac{2}{2-q}}}{2q(\alpha+\beta)(\alpha+\beta-2)^{\frac{q}{2-q}}}((|\lambda|||a||_{L^{p^{*}}})^{\frac{2}{2-q}}+(|\lambda|||a||_{L^{p^{*}}})^{\frac{2}{2-q}}).$$
This completes the proof.
\end{proof}
It follows from the above lemma that when $\lambda<\lambda_{1}$,
$(\int_{\Omega}a(x)|u|^{q}dx+\int_{\Omega}a(x)|v|^{q}dx)\\>0$ and
$(\int_{\Omega}b(x)|u|^{\alpha}|v|^{\beta}dx)>0$ then $\phi_{u,v}$
must have exactly two critical points as
discussed in the remarks preceding the lemma.\\
Thus when $\lambda<\lambda_{1}$ we have obtained a complete
knowledge of the number of critical points of $\phi_{u,v}$, of
the intervals on which $\phi_{u,v}$ is increasing and decreasing
and of the multiples of $(u,v)$ which lie in
$M_{\lambda}(\Omega)$ for every possible choice of signs of
$(\int_{\Omega}b(x)|u|^{\alpha}|v|^{\beta}dx)$ and
$(\int_{\Omega}a(x)|u|^{q}dx+\int_{\Omega}a(x)|v|^{q}dx)$. In
particular we have the following result.
\begin{corollary}
$M_{\lambda}^{0}(\Omega)=\emptyset$ when $0<\lambda<\lambda_{1}$.
\end{corollary}
By lemma \ref{maintheorem5}, for $0<|\lambda|<\lambda_{0}$ and $0<|\mu|<\mu_{0}$
we write $M_{\lambda,\mu}=M_{\lambda,\mu}^{+}\cup
M_{\lambda,\mu}^{-}$ and define
$$\alpha_{\lambda,\mu}^{+}=\displaystyle\inf_{(u,v)\in
M_{\lambda,\mu}^{+}} J_{\lambda,\mu}(u,v)$$and
$$\alpha_{\lambda,\mu}^{-}=\displaystyle\inf_{(u,v)\in
M_{\lambda,\mu}^{-}} J_{\lambda,\mu}(u,v)$$Then we have the
following results.\\\\
\begin{lemma}\label{maintheorem} There exist
minimizing sequences $\{(u_{n}^{\pm},v_{n}^{\pm})\}$ in
$M_{\lambda,\mu}^{\pm}$ for $J_{\lambda,\mu}$ such that
$J_{\lambda,\mu}(u_{n}^{\pm},v_{n}^{\pm})=\alpha_{\lambda,\mu}^{\pm}+o(1)$
and $J_{\lambda,\mu}^{'}(u_{n}^{\pm},v_{n}^{\pm})=o(1)$ in
$W^{'}$.
\end{lemma}
\begin{proof}[\bf Proof] The proof is almost the same as that in
Wu[\cite{7}. Proposition 9].
 \end{proof}

%------------------------------------------------------------------------------------------%

\section{\noindent\bf  Proof of Theorem 1.1}
 Throughout this section, we
assume that the parameters $\lambda$ and $\mu$ satisfies
$0<|\lambda|<\lambda_{0}$ and $0<|\mu|<\mu_{0}$. Then we have the
following results.
\begin{theorem}\label{maintheorem1} System $(E_{\lambda,\mu})$ has
a positive solution $(u_{0}^{+},v_{0}^{+})\in
M_{\lambda,\mu}^{+}$ such that
$J_{\lambda,\mu}(u_{0}^{+},v_{0}^{+})=\alpha_{\lambda,\mu}^{+}<0$.
\end{theorem}
\begin{proof}[\bf Proof]
 First, we show $\alpha_{\lambda,\mu}^{+}<0$. For $(u,v)\in
M_{\lambda,\mu}^{+}$, we have
$$\begin{array}{rcl}
J_{\lambda,\mu}(u,v)&=&(\frac{1}{p}-\frac{1}{q})||(u,v)||^{p}\\
&+&(\frac{1}{q}-\frac{1}{r+s})(\int_{\Omega}h|u|^{r}|v|^{s})\\
&<&-\frac{(p-q)(r+s-p)}{pq(r+s)}||(u,v)||^{p}<0.\\
\end{array}$$
This implies $\alpha_{\lambda,\mu}^{+}<0$. By lemma \ref{maintheorem},
there exists $\{(u_{n}^{+},v_{n}^{+})\}\subset
M_{\lambda,\mu}^{+}$ such that
$J_{\lambda,\mu}(u_{n}^{+},v_{n}^{+})=\alpha_{\lambda,\mu}^{+}+0(1)$
and $J_{\lambda,\mu}^{'}(u_{n}^{+},v_{n}^{+})=o(1)$ in $W^{'}$.
Then by lemma \label{maintheorem4} (ii) and the Rellich-Kondrachov theorem there
exist a subsequence $\{(u_{n}^{+},v_{n}^{+})\}$ and
$(u_{0}^{+},v_{0}^{+})\in W$ is a solution of system
$(E_{\lambda,\mu})$ such that
 $(u_{n}^{+},v_{n}^{+})\rightarrow(u_{0}^{+},v_{0}^{+})$ weakly in $W$ and
$(u_{n}^{+},v_{n}^{+})\rightarrow (u_{0}^{+},v_{0}^{+})$ strongly
in $L$ for all $1\leq l<p^{*}$.\\
Then we have
$$\int_{\Omega}f|u_{n}^{+}|^{q}+\int_{\Omega}g|v_{n}^{+}|^{q}=\int_{\Omega}f|u_{0}^{+}|^{q}+\int_{\Omega}g|v_{0}^{+}|^{q}+o(1)$$
and
$(\lambda\int_{\Omega}f|u_{0}^{+}|^{q}+\mu\int_{\Omega}g|v_{0}^{+}|^{q})\geq
0$.\\
Now, we prove that
$(\lambda\int_{\Omega}f|u_{0}^{+}|^{q}+\mu\int_{\Omega}g|v_{0}^{+}|^{q})>
0$ otherwise, then
$$||(u_{n}^{+},v_{n}^{+})||^{p}=\int_{\Omega}h|u_{n}^{+}|^{r}|v_{n}^{+}|^{s}+o(1)$$and
$$\begin{array}{rcl}
(\frac{1}{p}-\frac{1}{r+s})||(u_{n}^{+},v_{n}^{+})||^{p}_{W}&=&\frac{1}{p}||(u_{n}^{+},v_{n}^{+})||^{p}\\
&-&\frac{1}{q}(\lambda\int_{\Omega}f|u_{n}^{+}|^{q}+\mu\int_{\Omega}g|v_{n}^{+}|^{q})\\
&-&\frac{1}{r+s}\int_{\Omega}h|u_{n}^{+}|^{r}|v_{n}^{+}|^{s}+o(1)\\
&=&\alpha_{\lambda,\mu}^{+}+o(l)\\
\end{array}$$
This is contradicts $\alpha_{\lambda,\mu}^{+}<0$. Thus,
$(\lambda\int_{\Omega}f|u_{0}^{+}|^{q}+\mu\int_{\Omega}g|v_{0}^{+}|^{q})>
0$. In particular, $(u_{0}^{+},v_{0}^{+})\in M_{\lambda,\mu}^{+}$
is a nontrivial solution of system $(E_{\lambda,\mu})$ and
$J_{\lambda,\mu}(u_{0}^{+},v_{0}^{+})\geq\alpha_{\lambda,\mu}^{+}.$
Moreover,
$$\begin{array}{rcl}
\alpha_{\lambda,\mu}^{+}&\leq&J_{\lambda,\mu}(u_{0}^{+},v_{0}^{+})=(\frac{1}{p}-\frac{1}{q})(\lambda\int_{\Omega}f|u_{0}^{+}|^{q}+\mu\int_{\Omega}g|v_{0}^{+}|^{q})\\
&+&(\frac{1}{p}-\frac{1}{r+s})\int_{\Omega}h|u_{0}^{+}|^{r}|v_{0}^{+}|^{s}\\
&=&\displaystyle\lim_{n\rightarrow\infty}J_{\lambda,\mu}(u_{n}^{+},v_{n}^{+})=\alpha_{\lambda,\mu}^{+}.\\
\end{array}$$
Consequently,
$J_{\lambda,\mu}(u_{0}^{+},v_{0}^{+})=\alpha_{\lambda,\mu}^{+}$.
Since
$J_{\lambda,\mu}(u_{0}^{+},v_{0}^{+})=J_{\lambda,\mu}(|u_{0}^{+}|,|v_{0}^{+}|)$
and $(|u_{0}^{+}|,|v_{0}^{+}|)\in M_{\lambda,\mu}^{+}$. By lemma
\ref{maintheorem0} we may assume that $u_{0}^{+}\geq0,~v_{0}^{+}\geq0$.
Moreover, by the Harnack inequality due to Trudinger \cite{6}, we
obtain $(u_{0}^{+},v_{0}^{+})$ is a positive solution of system
$(E_{\lambda,\mu})$.
\end{proof}
\begin{theorem}\label{maintheorem2} System $(E_{\lambda,\mu})$ has a positive solution
$(u_{0}^{-},v_{0}^{-})\in M_{\lambda,\mu}^{-}$ such that
$J_{\lambda,\mu}(u_{0}^{-},v_{0}^{-})=\alpha_{\lambda,\mu}^{-}$.
\end{theorem}
\begin{proof}[\bf Proof] By lemma \ref{maintheorem}, there exists $\{(u_{n},v_{n})\}\subset
M_{\lambda,\mu}^{-}$ such that\\
$J_{\lambda,\mu}(u_{n}^{-},v_{n}^{-})=\alpha_{\lambda,\mu}^{-}+o(1)$
and $J_{\lambda,\mu}^{'}(u_{n}^{-},v_{n}^{-})=o(1)$ in $W^{'}$.\\
By lemma \ref{maintheorem8} and the Relich-Kondrachov  theorem, there exist a
subsequence $\{(u_{n}^{-},v_{n}^{-})\}$ and
$(u_{0}^{-},v_{0}^{-})\in M_{\lambda,\mu}^{-}$ is a nonzero
solution of system $(E_{\lambda,\mu})$ such that
 $(u_{n}^{-},v_{n}^{-})\rightarrow(u_{0}^{-},v_{0}^{-})$ weakly in $W$ and
$(u_{n}^{-},v_{n}^{-})\rightarrow (u_{0}^{-},v_{0}^{-})$ strongly
in $L$. Moreover,
$$\begin{array}{rcl}
\alpha_{\lambda,\mu}^{-}&\leq&J_{\lambda,\mu}(u_{0}^{-},v_{0}^{-})=(\frac{1}{p}-\frac{1}{q})(\lambda\int_{\Omega}f|u_{0}^{-}|^{q}+\mu\int_{\Omega}g|v_{0}^{-}|^{q})\\
&+&(\frac{1}{p}-\frac{1}{r+s})\int_{\Omega}h|u_{0}^{-}|^{r}|v_{0}^{-}|^{s}\\
&=&\displaystyle\lim_{n\rightarrow\infty}J_{\lambda,\mu}(u_{n}^{-},v_{n}^{-})=\alpha_{\lambda,\mu}^{-}.\\
\end{array}$$
Consequently,
$J_{\lambda,\mu}(u_{0}^{-},v_{0}^{-})=\alpha_{\lambda,\mu}^{-}$.
Since
$J_{\lambda,\mu}(u_{0}^{-},v_{0}^{-})=J_{\lambda,\mu}(|u_{0}^{-}|,|v_{0}^{-}|)$
and $(|u_{0}^{-}|,|v_{0}^{-}|)\in M_{\lambda,\mu}^{-}$. By lemma \ref{maintheorem0}
 we may assume that $u_{0}^{-}\geq0,~v_{0}^{-}\geq0$.\\
Moreover, by the Haranack inequality due to Trudinger \cite{6}, we
obtain $(u_{0}^{-},v_{0}^{-})$ is a positive solution of system
$(E_{\lambda,\mu})$.\\
\end{proof}

Now, we begin to show the proof of {\bf theorem 1.1}: By theorem
\ref{maintheorem1}, \ref{maintheorem2} system $(E_{\lambda,\mu})$ has two positive solutions
$(u_{0}^{+},v_{0}^{+})$ and $(u_{0}^{-},v_{0}^{-})$ such that
$(u_{0}^{+},v_{0}^{+})\in M_{\lambda,\mu}^{+}$ and
$(u_{0}^{-},v_{0}^{-})\in M_{\lambda,\mu}^{-}$. Since
$M_{\lambda,\mu}^{+}\cap M_{\lambda,\mu}^{-}=\emptyset$, this
implies that $(u_{0}^{+},v_{0}^{+})$ and $(u_{0}^{-},v_{0}^{-})$
are distinct.\\\\

%----------------------------------------------------------------------------------------%

\end{document}